\documentclass{amsart}

\DeclareMathOperator{\Spec}{Spec}

\DeclareMathOperator{\prim}{prim}

\newcommand{\mbb}[1]{\mathbb{#1}}
\newcommand{\uHom}{\underline{\mathrm{Hom}}}

\newcommand{\Z}{\mbb{Z}}
\newcommand{\Q}{\mathbb{Q}}
\newcommand{\F}{\mathbb{F}}

\usepackage{color}

\newcommand{\Preston}[1]{}

\usepackage{amscd}
\usepackage{amssymb}
\usepackage{hhline}

\numberwithin{equation}{section}

\newtheorem{theorem}{Theorem}[section]

\newtheorem{lemma}[theorem]{Lemma}

\theoremstyle{definition}
 
\newtheorem{remark}[theorem]{Remark} 
\newtheorem{example}[theorem]{Example}

\usepackage{hyperref}

\title{Primitive elements for $p$-divisible groups}

\author{Robert Kottwitz}
\address{Dept. of Mathematics \\
University of Chicago \\
5734 S. University Avenue \\
Chicago, Illinois 60637}
\email{kottwitz@math.uchicago.edu}

\author{Preston Wake}
\address{UCLA Mathematics Department \\
Box 951555 \\
Los Angeles, CA 90095-1555 }
\email{wake@math.ucla.edu}
\urladdr{math.ucla.edu/~wake/} 

\begin{document} 

\begin{abstract}
We introduce the notion of primitive elements in arbitrary truncated $p$-divisible groups. By design, the scheme of primitive elements is finite and locally free over the base. Primitive elements generalize the ``points of exact order $N$,'' developed by Drinfeld  and Katz-Mazur  for elliptic curves.
\end{abstract}

\maketitle

\section{Introduction}

In this paper, we observe that Raynaud's theory of Haar measures on finite flat group schemes \cite{R} may be used to define a ``non-triviality" condition on sections, which we call \emph{non-nullity}. For groups of order $p$, we show that non-null sections are ``generators" in the sense of Oort-Tate theory \cite{OT}. For truncated $p$-divisible groups, we use a non-nullity condition to define the notion of \emph{primitivity}, 
generalizing the ``points of exact order $N$'' of Drinfeld \cite{D} and Katz-Mazur \cite{KM}.

In the case of elliptic curves, Drinfeld and Katz-Mazur go further and define full level structures. This allows them to construct and prove nice properties of integral models of modular curves at arbitrary levels in a very elegant fashion. 
We believe that our definition of primitive elements may be a first step toward defining full level structures 
in certain cases, as it was in previous work by one of us in the case $\mu_p \times \mu_p$ \cite{W}. However, for general $p$-divisible groups, we believe that new ideas are needed, and we hope that this work will lead to a better understanding of the issues involved in defining full level structures.

\subsection{The problem of full level structures} To understand the problem of finding level structures, consider the following setup. Let $S$ be a Noetherian scheme that is flat over $\Z_{(p)}$, and let $G$ be a finite flat group scheme such that $G[1/p]:=G \times_S S[1/p]$ is \'etale-locally isomorphic to $(\Z/p^r\Z)^g$ (for instance, $S$ could be a  Shimura variety classifying $g/2$-dimensional abelian varieties with additional structure, and $G$ could be the $p^r$-torsion of the universal abelian variety). A level structure on $G$ is a map $(\Z/p^r\Z)^g \to G$ that is like an isomorphism. The desired properties of level structures are best described scheme-theoretically. The set of full level structures $\mathcal{F}_G$ should be a closed subscheme of $\uHom_S \bigl( (\Z/p^r\Z)^g,G \bigr)$ satisfying:
\begin{list}{$\bullet$}{}
\item  $\mathcal{F}_G$ is flat over $S$
\item $\mathcal{F}_G \times_S S[1/p] = \underline{\mathrm{Isom}}_{S[1/p]}\bigl( (\Z/p^r\Z)^g,G[1/p] \bigr)$.
\end{list}
Since $\uHom_S \bigl( (\Z/p^r\Z)^g,G \bigr)$ is flat over $\Z_{(p)}$, these conditions determine $\mathcal{F}_G$ uniquely. However, in practice, it may be difficult to tell if a given homomorphism is full. For many purposes, $\mathcal{F}_G$ is only useful if there is an explicit description of the ideal defining it.

\subsection{Previous results} In the case where $G$ embeds into a smooth curve $C$ over $S$ (for example if $G=E[p^r]$ for an elliptic curve $E$), a satisfactory theory of full level structures has been built out of the ideas of Drinfeld \cite{D}. However, Drinfeld's definition crucially uses the fact that $G$ is a Cartier divisor in $C$. Katz and Mazur  developed a notion of ``full set of sections,'' which they show is equivalent to the Drinfeld level structure in the case that $G \subset C$ 
\cite[\S1.10]{KM}. However, as Chai and Norman pointed out \cite[Appendix]{CN}, the Katz-Mazur definition does not give a flat space in general -- it fails even for the relatively simple example of $G=\mu_p \times \mu_p$. 

More recently, one of the present authors developed a notion of full homomorphisms in the specific case $G=\mu_p \times \mu_p$ \cite{W}. 

\subsection{Primitive elements} The first step in finding a basis for a free module is to find a primitive vector -- that is, an  element that can be extended to a basis. Analogously, a
first step towards defining a notion of full level structure might be
to define a notion of primitive element for group schemes. In addition, the notion of primitive element is needed to define the correct notion of ``linear independence,'' which is a key part of the method in \cite{W} for $G=\mu_p \times \mu_p$. In this paper we develop a formal theory of primitive elements, generalizing the ad hoc notion defined in \cite{W}.

\subsection{Primitive elements and full homomorphisms}
One may suggest defining a homomorphism $\varphi:(\Z/p^r\Z)^g \to G$ to be ``full" if it sends primitive vectors to primitive vectors. Indeed, if $G$ is constant, then this corresponds to the condition that the matrix of $\varphi$ has linearly independent columns. However, the example of $\mu_p \times \mu_p$ studied in \cite{W} shows why this definition does not give a flat space of full homomorphisms. In that case, one may think of $\varphi$ as a ``$2\times 2$-matrix with coefficients in $\mu_p$.'' If $\varphi$ sends primitive vectors to primitive vectors, then the columns are ``linearly independent,'' but the rows may not be -- hence the elements cutting out the condition that the rows be ``linearly independent" are $p$-torsion elements in the coordinate ring of the space of full homomorphisms. On the other hand, the main theorem of \cite{W} implies that column conditions together with the row conditions give a flat space. 

For a general group $G$, there is no obvious analog of the row conditions, so it is not clear how to generalize from primitive vectors to full homomorphisms. A new idea is needed. 


\subsection{Summary}  Let $S$ be a scheme, and let $G$ be a finite locally free (commutative) group scheme over $S$. Let $|G|$ denote the rank of $G$. We define a closed subscheme $G^\times \subset G$, which we call the \emph{non-null} subscheme. The ideal cutting out $G^\times$  consists of invariant measures, 
as in Raynaud's theory \cite{R}, on the Cartier dual of $G$. As a consequence of Raynaud's results, $G^\times$ is finite and locally free over $S$ of rank $|G|-1$. We think of $G^\times$ as the  group-scheme version of the set of non-zero elements of $G$. 

There does not seem to be any completely satisfactory word to use here. Since the identity element 
in $G(S)$ can perfectly well lie in $G^\times(S)$ (as happens in the second example below when $S$ is a scheme over 
$\mathbb F_p$), it would be extremely confusing say that elements in $G^\times(S)$ are non-zero. Instead we have chosen to 
say that they are ``non-null.''

As evidence that the notion of non-nullity is reasonable, 
 we mention 
the following examples:
\begin{list}{$\bullet$}{•}
\item If $G=\Gamma_S$, the constant group-scheme associated to a finite abelian group $\Gamma$, then $G^\times = (\Gamma \setminus \{0\})_S$, the scheme of non-identity sections.
\item If $G=\mu_p$, then $G^\times =\mu_p^\times$, the scheme of primitive roots of unity.
\item If $G$ is an Oort-Tate group \cite{OT} (i.e.~$|G|=p$), then $G^\times$ coincides with the \emph{scheme of generators} defined by Haines and Rapoport \cite{HR}.
\item If $G$ is a Raynaud group \cite{R} (i.e.~ $G$ has an action of $\mathbb{F}_q$ and $|G|=q$ for some power $q$ of $p$), then $G^\times$ coincides with the \emph{scheme of $\mathbb{F}_q$-generators} defined by Katz-Mazur (c.f.~ \cite{pappas}) 
\end{list}

We define primitive elements using non-nullity as follows. Assume that $G=\mathcal{G}[p^r]$, the $p^r$-torsion subgroup of a $p$-divisible group $\mathcal{G}$ of height $h$. Let $G^{\prim} = G \times_{\mathcal{G}[p]} (\mathcal{G}[p])^\times$, where the map $G \to \mathcal{G}[p]$ is given by multiplication by $p^{r-1}$. It follows that the subscheme $G^{\prim} \subset G$ is locally free over $S$ of rank $(p^h-1)p^{h(r-1)}$.

In specific examples, we can identify $G^{\prim}$:
\begin{list}{$\bullet$}{•}
\item If $V=(\Q_p/\Z_p)^h$ and $\mathcal{G}=V_S$ is a constant $p$-divisible group, then $G^{\prim}$ is the scheme associated to the set of primitive vectors in the free $\Z/p^r\Z$-module $V[p^r]$.
\item If $G=\mu_{p^r}$, then $G^{\prim}$ is the subscheme of primitive roots of unity.
\item If $G=E[p^r]$ for an elliptic curve $E$, then $G^{\prim}$ is the scheme of sections ``of exact order $p^r$'' defined by Drinfeld-Katz-Mazur \cite{D, KM}.
\end{list}

This justifies the notation $G^{\prim}$ -- it is meant to evoke both the notion of primitive vector in a free module, and primitive root of unity.

\subsection{Applications to Shimura varieties}  Let $X$ be a Shimura variety over $\Q$ that has a universal abelian variety $A$ over it, and suppose $\mathfrak{X}$ and $\mathcal{A}$ are models for $X$ and $A$ that are flat over $\Z_{(p)}$. Then, for each $r>1$, there is an interesting cover $X_1(p^r)$ of $X$ given by adding the additional data of a point of order $p^r$ in $A$.

The scheme $\mathfrak{X}_1(p^r):=\mathcal{A}[p^r]^{\prim}$ is an integral model for $X_1(p^r)$ that is finite and flat over $\mathfrak{X}$. Since $\mathfrak{X}$ is flat over $\Z_{(p)}$, this implies that $\mathfrak{X}_1(p^r)$ is the Zariski-closure of $X_1(p)$ in $\uHom_\mathfrak{X}(\Z/p^r\Z,\mathcal{A}[p^r])$. In particular, this ``flat-closure" model, which is a priori only flat over $\Z_{(p)}$, is actually flat over $\mathfrak{X}$.

On the other hand, one can show that, except for modular curves (or the Drinfeld case), the scheme $\mathfrak{X}_1(p^r)$ is not normal. In particular, $\mathfrak{X}_1(p^r)$ is not the normalization of $\mathfrak{X}$ in $X_1(p^r)$, and this gives an example where the ``normalization" and ``flat closure" models differ. 

This issue of non-normality makes us doubtful that these models will have direct application to the Langlands program. Instead, we view the theory of primitive elements as an interesting tool to use in the future study of integral models. For example, it would be interesting to consider combining the notion of primitive element with parahoric models of Shimura varieties, in analogy with the work of Pappas on Hilbert modular varieties \cite{pappas}. Using the theory of Raynaud group schemes, Pappas produces a model for $\Gamma_1(p)$-type level that is normal (but not finite over the base). \Preston{Added some remarks about normal models using Raynaud's theory}

\subsection{Acknowledgements} 
We thank G.~Boxer, B.~Levin and K.~Madapusi Pera for interesting conversations about integral models. We are grateful to T.~Haines, G.~Pappas, and M.~Rapoport for helpful comments on a preliminary version of this paper. We thank the referees for comments and suggestions.

\section{Review of Raynaud's Haar measures for finite flat group schemes} 

In this section we work over an affine base scheme $S = \Spec(k)$, and 
$G$  denotes a commutative group scheme over $S$ that is  finite, flat and
finitely presented. So $G = \Spec(A)$ with $A$ locally free of finite rank as $k$-module. 
This rank is a locally constant function, denoted $|G|$, on $S$.

We write $G'=\Spec(A')$ for the Cartier dual of $G$; it is another object of the same kind 
as $G$, and $|G'| = |G|$. Recall that $A$ and $A'$ are the $k$-duals of 
each other.

As Raynaud \cite{R} points out,  it is helpful to think about $f \in A$ as a function on $G$ and 
$\mu \in A'$ as a measure on $G$, and then to write $\langle \mu, f \rangle \in k$ 
for the natural pairing of $\mu$ with $f$. Closely following Raynaud's notation and conventions, we 
\begin{itemize}
\item
write $\star$ for the 
multiplication law on $A'$ (intuitively, convolution of measures), 
\item 
write $1$ for the unit element in the ring $A$ (intuitively, the constant function with value $1$), 
\item 
write  $\delta$ for the unit element in the ring $A'$, i.e.~the counit $A \to k$ for 
the coalgebra $A$ (intuitively, evaluation of functions at the identity element in 
the group $G$), 
\item 
 denote the natural $A$-module structure on $A'$ by $f \mu$ (intuitively, pointwise multiplication of 
a measure by a function), and 
\item 
 denote the natural 
$A'$-module structure on $A$ by $\mu \star f$ (intuitively, the convolution of a function by a measure). 
\end{itemize}
By definition these actions are given by  the  formulas 
\[
\langle f \mu, g \rangle = \langle \mu, fg \rangle,  \qquad  \langle \nu, \mu \star f \rangle =
 \langle \mu \star \nu, f \rangle.
\]

A $G$-module  is by definition a comodule for the coalgebra $A$, but,   
because $A$ is locally free of finite rank as $k$-module, giving a $G$-module is the same as giving an 
 $A'$-module $M$. For example, the $A'$-module structure on $A$ reviewed above is the one corresponding 
to the natural $G$-module structure on $A$. 

Given a $G$-module $M$, 
its submodule $M^G$ of $G$-invariants  consists of all elements in $M$ annihilated by the augmentation 
ideal $I'$ in $A'$.  For any $k$-algebra $R$ there is a natural map 
\begin{equation} \label{eq.InvBC}
(M^G) \otimes R \to (M \otimes R)^{G \otimes R}.
\end{equation}
(We are abbreviating $\otimes_k$ to $\otimes$.)
When \eqref{eq.InvBC} is an isomorphism for every $k$-algebra, one says that forming $G$-invariants in $M$ 
commutes with extension of scalars. Bear in mind that  $M$ need not have this property, even when it 
is locally free of finite rank as $k$-module. 

For the $G$-module $A$ one has $A^G = k$. So, in this example, it is evident 
that forming invariants does 
commute with extension of scalars.  
Now $A'$ is of course an $A'$-module, i.e.~a $G$-module, so we can form its submodule 
of $G$-invariants
\[
D_G := (A')^G = \{ \mu \in A' : \nu \star \mu =0 \quad \forall\, \nu \in I'  \}.
\] 
We will refer to elements of $D_G$ as \emph{$G$-invariant measures} on $G$. 
From the decomposition $A'=k \oplus I'$ it follows immediately that $D_G$ can also be described as  
\[
 \{ \mu \in A' : \nu \star \mu = \langle \nu,1  \rangle \mu \quad \forall\, \nu \in A'  \}.
\]

Raynaud proves (in the discussion on page 277 of \cite{R}) that $(A')^G$ is a direct summand of $A'$, locally 
free of rank $1$ as $k$-module. In other words, $G$-invariant measures on $G$ form a line bundle over 
$S$. When $D_G$ is free of rank $1$ (not just locally so), a basis element $\mu$ for it is called a 
\emph{Haar measure on $G$}. A $G$-invariant measure $\mu:A \to k$ is a Haar measure if and only 
if it is surjective. 

The line bundle of $G'$-invariant measures on $G'$ is then 
\begin{equation}
\label{eq: defn of J_G}
J_G := A^{G'} = \{ f \in A : gf=\langle \delta, g \rangle f \quad \forall\, g \in A \}, 
\end{equation}
a direct summand of $A$ that is locally free of rank $1$ as a $k$-module. Note that since $A=I\oplus k$, where $I=\ker(\delta)$ is the augmentation ideal, \eqref{eq: defn of J_G} implies that $J_G$ is the annihilator of $I$ in $A$.  Raynaud also proves that 
\begin{enumerate}
\item
The natural pairing $A' \otimes A \to k$ restricts to a perfect pairing between $D_G$ and $J_G$. So 
the line bundles $D_G$ and $J_G$ on $S$ are canonically dual to each other. 
\item The map $f \otimes \mu \mapsto f\mu$ is an isomorphism  $A \otimes D_G \to A'$. This map 
is an isomorphism of $A$-modules and of $A'$-modules, with $g \in A$ (respectively, $\nu \in A'$) 
operating on $A \otimes D_G$ by the rule $g(f \otimes \mu):= (gf) \otimes \mu$ (respectively, 
$\nu \star(f \otimes \mu) := (\nu \star f) \otimes \mu$). 
\end{enumerate} 

It follows from (2) that, Zariski locally on $S$, the $G$-module $A'$ is isomorphic to the $G$-module $A$. 
So forming $G$-invariants in $A'$ (i.e. forming $D_G$) commutes with extension of scalars. (This 
useful fact is brought out  by Moret-Bailly in the section of \cite{MB} in which he summarizes Raynaud's work.)

\subsection{Integration in stages}
\label{subsec: integration}
We need one more fact about Haar measures, namely an analog of the ``integration in stages formula'' 
in the theory of Haar measures on locally compact groups. It seems plausible that this is well-known, but 
since we do not know a reference, we will provide a proof. 

Consider a short exact sequence 
\[
0 \to H \to G \to K \to 0. 
\]
Here $H = \Spec(B)$, $K=\Spec(C)$ are objects of the same type as $G$, so $B$ and $C$ are locally  
free $k$-modules.  Moreover $A$ is faithfully flat 
over its subalgebra $C$, and $B$ is the quotient of $A$ by the ideal generated by the augmentation 
ideal $I_C$ in $C$.  Our goal is to understand invariant measures on $G$ in terms of 
invariant measures on $H$ and $K$. Dual to $C \subset A$ and $A \twoheadrightarrow B$ are 
the algebra homomorphisms 
\[
A' \twoheadrightarrow C' \qquad B' \subset A'. 
\]
Observe that the kernel of $A' \twoheadrightarrow C'$ is the ideal in $A'$ generated 
by the augmentation ideal $I_{B'} = (I_B)'$ in $B'$.

\begin{lemma} \label{lem.IntStages}
Let $\mu_H \in D_H$ and  $\mu_K \in D_K$. Choose  $\tilde{\mu}_K  \in A'$ mapping 
to $\mu_K$ under $A' \twoheadrightarrow C'$. 
Then the element $\mu_G \in A'$ defined by $\mu_G = \tilde{\mu}_K \star \mu_H$ lies in $D_G$ 
and is independent of the choice of $\tilde{\mu}_K$. 
Moreover the map $\mu_K \otimes \mu_H \mapsto \mu_G$ is an isomorphism 
$D_K \otimes D_H \to D_G$. 
 \end{lemma}

\begin{proof}
It is evident that $\mu_G$ is independent of the choice of the lifting $\tilde{\mu}_K$, because this 
lifting is well-defined modulo $(I_{B'}) A'$, and $I_{B'}$ annihilates $\mu_H$. The rest of the lemma 
is most easily  understood in terms of integration in stages, as we will now see. 

For any $H$-module $M$ the invariant measure $\mu_H$ gives rise to a $k$-linear map 
$M \to M^H$, defined by $m \mapsto \mu_H \star m$. (We use $\star$ to denote the operation 
of an element in $B'$ on an $H$-module.)  
Applying this to the $H$-module $A$, we obtain 
 a  $k$-linear map 
\[
\mathcal I_H : A \to A^H=C, 
\]
 given by convolution with $\mu_H$ (intuitively, integration over the orbits 
of $H$ on $G$). We claim that, if $\mu_H$ is a Haar measure, then $\mathcal I_H$ is surjective. 
Indeed, this is a special case of the following more general statement. Let $T=\Spec(D)$ be an affine $S$-scheme, and 
let $X=\Spec(E)$ be an $H$-torsor over $T$. 
      Then the map $\mathcal I_H : E \to D$ (given by convolution with the Haar measure $\mu_H$) 
is surjective. Surjectivity of $\mathcal I_H$ is fpqc local,
 so we are reduced to the case in which $X = H \times T$. Then $\mathcal I_H$ is 
obviously surjective, because it is obtained by tensoring $\mu_H : B \twoheadrightarrow k$ with $D$.

The map $\mathcal I_H : A \to C$ is equivariant with respect to $G \twoheadrightarrow K$ (and the natural actions of 
$G$ on $A$ and $K$ on $C$), and the composition  $\mu:=\mu_K \circ \mathcal I_H : A \to k$ is  $G$-equivariant, 
i.e.~$\mu \in D_G$.   Unwinding the definitions, one sees that $\mu = \mu_G$. 
The work we did shows that $\mu_G$ is surjective when both $\mu_H$, 
$\mu_K$ are surjective, and hence that 
 $\mu_K \otimes \mu_H \mapsto \mu_G$ is an isomorphism from 
$D_K \otimes D_H$ to $D_G$. 
\end{proof}

\section{Non-null elements in $G$}

In this section we continue with $k$ and $G$ as in the previous section. 

\subsection{Definition of non-nullity of elements in $G$}\label{sub.DefPrim} 
The explicit description \eqref{eq: defn of J_G} of $J_G$ shows that it is an ideal in $A$. We will refer 
to $J_G$ as the \emph{non-nullity ideal}. 
 We denote by $G^{\times} \hookrightarrow G$ the closed subscheme of $G$ cut out by the ideal $J_G$. Observe that  $G^{\times}=\Spec(A/J_G)$ 
is  locally free of rank $|G|-1$ over $S$. 

For every $k$-algebra $R$, $G^{\times}(R)$ is a subset of $G(R)$. 
We say that an element $g \in G(R)$ is \emph{non-null} when it lies in the subset $G^{\times}(R)$.  In the next subsections we will investigate this notion.

\subsection{Non-nullity in the constant case} 

Start with a finite abelian group $\Gamma$ and use it to build a constant group scheme $G/S$. Then $A$ is the algebra 
of $k$-valued functions $f:\Gamma \to k$, the ring structure being pointwise multiplication of functions. Then
 $I_G=\{ f \in A : f(e_\Gamma) =0 \}$ and $J_G=\{ f \in A : f(\gamma) =0 \quad \forall\, \gamma \ne e_\Gamma \}$ \Preston{changed $g$ to $\gamma$ in the definition of $J_G$ for clarity ($g$ was used as an element of $A$ before)}.  
So $A = J_G \oplus I_G$. In other words the scheme $G$ decomposes as the disjoint union of two open (and 
closed) subschemes: $G^{\times}$ and the identity section $e_G(S)$. This example explains why we have 
chosen to call  $G^{\times}$ the closed subscheme 
of non-null elements in $G$. 

\subsection{Testing non-nullity using an overring $R' \supset R$}  \Preston{Moved this subsection to an earlier location}
An $R$-valued point of $G$ is given by a $k$-algebra homomorphism $g:A \to R$. The element 
$g \in G(R)$ is non-null if and only if the ring homomorphism $g:A \to R$ is $0$ on the 
ideal $J_G$ in $A$. Consequently, if $R \to R'$ is an injective $k$-algebra 
homomorphism, then $g \in G(R)$ is non-null if and only if its image in $G(R')$ is non-null. 

If $R'/R$ is faithfully flat, then $R \to R'$ is injective. So the notion of  non-nullity is 
fpqc local, and therefore continues to make sense for any base scheme (or even algebraic space) $S$. 

\subsection{Non-nullity in the \'etale case} 

Assume that $G/S$ is \'etale. Then, locally in the \'etale topology, $G$ is constant. It follows from the 
calculation in the previous subsection that $A$ is the direct sum of the ideals $J_G$ and $I_G$. In other words,
 $A$ is the cartesian product of the $k$-algebras $A/J_G$ and $A/I_G$. Therefore 
\begin{itemize}
\item
the closed subscheme $G^{\times}$ is also an open subscheme of $G$,  
\item 
 the identity section $e_G : S \hookrightarrow G$ is an open and closed immersion, and
\item 
 $G$ decomposes as the disjoint union 
\[ 
G = G^{\times} \, \coprod \, e_G(S).
\] 
of open subschemes.
\end{itemize}
 So, in the \'etale case, $G^{\times}$ is again the open (and closed) subscheme of $G$ obtained by 
deleting the image of the identity section $S \to G$. 

\subsection{Behavior under base change} \label{sub.BaseChange}

Consider a $k$-algebra $R$. For any scheme $X/k$ we denote by $X_R$ its base change to $R$. 
In particular we may base change $G$ to $R$, obtaining a group scheme $G_R=\Spec(R \otimes A)$ over $R$. 
In our review of Haar measures, we mentioned that the natural map $R \otimes J_G \hookrightarrow J_{G_R}$
is an isomorphism, which tells us that 
 the natural morphism  
\[
(G_R)^{\times} \to (G^{\times})_R
\]
is an isomorphism. In other words, forming $G^{\times}$ from $G$ commutes with extension of scalars. 

%
%
%

\subsection{Non-nullity for Oort-Tate groups}

Now let us examine the notion of non-nullity in the case of Oort-Tate groups \cite{OT}.  
Our notion of non-nullity applies to all 
groups of order $p$ over any base ring $k$, but in order to compare it 
to the notion of Oort-Tate generator we need to restrict attention to $\Lambda$-algebras,
where $\Lambda$ is the base ring considered in \cite{OT}. If $\zeta \in \Z_p$ is a primitive $(p-1)$-rst root of unity, then
\[
\Lambda = \Z\left[\zeta,\frac{1}{p(p-1)}\right] \cap \Z_p
\]
with the intersection taking place in $\Q_p$.


Let $k$ be a $\Lambda$-algebra (e.g., a $\mathbb Z_p$-algebra). 
Then, given suitable $a,b \in k$, Oort-Tate construct 
a group $G_{a,b}$ of order $p$ over $k$, but we will fix $a,b$ 
and just call the group $G$.  The corresponding $k$-algebra  is 
 $A = k[x]/(x^p-ax)$, and its augmentation ideal   is  
generated by $x$. So the ideal $J_G$ consists of all elements in $A$ that are annihilated by $x$, and 
a short computation reveals that $J_G$ is the $k$-submodule of $A$ generated 
by $x^{p-1}-a$. 
This shows that an element $g \in G(k)$ is non-null in our sense if and only if $g$ is a generator of $G$ in the sense of Haines-Rapoport \cite{HR} (this notion of generator was first used by Deligne-Rapoport in \cite[Section V.2.6, pg.~ 106]{Deligne-Rapoport}) \Preston{added a reference to Deligne-Rapoport}. Moreover, as Haines-Rapoport show \cite[Remark 3.3.2]{HR}, this is also equivalent to $g$ having ``exact order $p$'' in the sense of Drinfeld-Katz-Mazur. This agreement suggests that the notion of non-nullity is a natural one. 

\begin{example}
The above discussion applies to the group $\mu_p$ of $p$-th roots of unity. The result is that a section $\zeta \in \mu_p(k)$ lies in $\mu_p^{\times}$ if and only if $\Phi_p(\zeta)=0$ (where $\Phi_p(T) = 1 + T + \dots + T^{p-1}$ is the cyclotomic polynomial). In other words, $\mu_p^\times$ is the subscheme of primitive $p$-th roots of unity. 
\end{example}

\subsection{Non-nullity for Raynaud groups} Raynaud groups are a natural generalization of Oort-Tate groups, and in this case, again, the notion of non-nullity agrees with a well-studied notion. We thank G.~Pappas for communicating this generalization to us.

Let $q=p^n$ be a power of $p$ and let $D$ be the ring defined analogously to $\Lambda$, but with $q$ in place of $p$ (see \cite[Section 1.1]{R}), and let $k$ be a $D$-algebra. Given a suitable $2n$-tuple $(\delta_1,\dots,\delta_n,\gamma_1,\dots,\gamma_n) \in k^{2n}$, Raynaud, in \cite[Collolaire 1.5.1]{R}, defines a group scheme $G$ over $k$ with $|G|=q$ together with an action of $\F_q$ on $G$ -- that is, $G$ is an $\F_q$-vector space scheme of dimension 1. The corresponding $k$-algebra is $A=k[x_i]/(x_i^p-\delta_{i}x_{i+1})$ where $i$ ranges over $\{1,\dots,n\}$ and $x_{n+1}:=x_1$. Then the augmentation ideal is generated by $(x_1,\dots,x_n)$, and using \cite[Proposition 2.1]{lci}, for example, one can see that $J_{G}$ is the $k$-submodule of $A$ generated by $(x_1\cdots x_n)^{p-1} - \delta_1 \dots \delta_n$. By \cite[Proposition 5.1.5]{pappas}, $G^\times$ is the scheme of ``$\F_q$-generators of $G$'', in the sense of Katz-Mazur.

\subsection{Products}

Consider groups $G_1$, $G_2$ over $k$. The corresponding $k$-algebras, augmentation ideals, 
and non-nullity ideals will be denoted $A_i$, $I_i$, $J_i$ (for $i = 1,2$).  
The ring of regular functions for the group $G=G_1 \times G_2$ is 
 $A = A_1 \otimes A_2$, and its augmentation ideal $I_G$  is 
$(I_1 \otimes A_2) + (A_1 \otimes I_2)$. Therefore the ideal $J_G$ in $A$ annihilated by $I$ is the 
intersection of the ideal annihilated by $I_1$, namely $J_1 \otimes A_2$, and the one annihilated 
by $I_2$, namely $A_1 \otimes J_2$. (It follows that
$
J_G= J_1 \otimes J_2  ,
$
and so $J_G$ is also the product of the ideals  $J_1 \otimes A_2$ and $A_1 \otimes J_2$.)

In more geometrical language, we just verified that $G^\times$ is the ``union'' of the closed subschemes 
$G_1 \times G_2^\times$ and $G_1^\times \times G_2$ of $G$ (i.e.~it is the smallest closed 
subscheme containing the two given closed subschemes). 

Some care is required in this situation. Consider an $R$-valued point $g=(g_1,g_2)$ of $G$. 
 If $g_1$ is non-null or $g_2$  is non-null, then $g$ is non-null. However, the converse is false, 
as is illustrated by the next example (when considering points with values in a ring that is not an integral domain).

\begin{example}
 Let $(x,y) \in \mu_p \times \mu_p$. Then $(x,y)$ is non-null if and only if $\Phi_p(x)\Phi_p(y)$ vanishes. 
So, for this group, non-nullity coincides with the notion of primitivity introduced  in \cite{W}. 
\end{example}

\subsection{Extensions}

Consider a short exact sequence 
\[
0 \to H \xrightarrow{i} G \xrightarrow{\pi} K \to 0
\]
as in Section \ref{subsec: integration}. We use the same system of notation: to $G$, $H$, $K$ correspond  $k$-algebras $A$, $B$, $C$ respectively. Their augmentation ideals 
will be denoted $I_A$, $I_B$, $I_C$, and their non-nullity ideals will be denoted $J_A$, $J_B$, $J_C$. 
Applying   Lemma \ref{lem.IntStages} to the Cartier dual of $G$, we see that  $J_B \otimes J_C \simeq J_A$, just as 
in the special case when $G=H \times K$.   In fact  Lemma \ref{lem.IntStages} says more. It tells us that 
\begin{equation}\label{eq.JABC}
J_A=(i^*)^{-1}(J_B) J_C, 
\end{equation}
 where $i^*$  denotes the  surjection $A \twoheadrightarrow B=A/I_C A$ 
obtained from $i: H \hookrightarrow G$. 
From this we obtain the following lemma.  
\begin{lemma}
The closed subscheme $G^{\times}$ of $G$ contains both of the following closed subschemes of 
$G$: 
\begin{itemize}
\item
the closed subscheme $H^{\times} \hookrightarrow H \hookrightarrow G$,
\item 
the closed subscheme $\pi^{-1}(K^\times)$ of $G$ obtained from $K^{\times} \hookrightarrow K$ by base change along 
$\pi : G \twoheadrightarrow K$.
\end{itemize}
The first item tells us that there exists an  arrow $j$ making 
\begin{equation}\label{CD.HGprim}
\begin{CD}
H^{\times} @>j>> G^{\times} \\
@VVV @VVV \\
H @>>> G
\end{CD}
\end{equation} 
commute. This arrow is  unique, and it is a closed immersion. 
If $K$ is \'etale over $S$, then the first item can be strengthened to the statement that the square 
\eqref{CD.HGprim} 
is cartesian. 
\end{lemma} 

\begin{proof}
We begin with the first item. The first item is true if and only if there exists an arrow $j$ making 
the square \eqref{CD.HGprim} commute. This is the condition that $i^*(J_A) \subset J_B$. 
That this condition holds follows from \eqref{eq.JABC}, which shows that $i^*(J_A)$ is the product 
of the ideals $J_B$ and $(i^*)(J_C)$ in $B$. 

The first item can be strengthened to the statement that the square \eqref{CD.HGprim} is cartesian if and 
only if the inclusion  $i^*(J_A) \subset J_B$ is an equality. This is certainly the case when $i^*(J_C)$ 
is the unit ideal in $B$.

When $K$ is 
\'etale over $S$, we have seen that $C = J_C \oplus I_C$. Therefore there exists $f \in J_C$ such 
that $1-f \in I_C$. The image of $f$ under $i^*$ is equal to $1$, showing that  $i^*(J_C)$ is indeed 
the unit ideal in $B$. 

Finally, the second item is true if and only if the ideal $AJ_C$  contains the ideal $J_A$. The truth of this is obvious from \eqref{eq.JABC}. 
\end{proof}

\begin{remark} Let $h \in H(R)$. 
The lemma implies that, if $h$ is non-null for $H$, then it is non-null for $G$.  It also implies that the converse 
is true provided that $K$ is \'etale over $S$. In general the converse is false. For example,  $(1,y) 
\in \mu_p(R) \times \mu_p(R)$ is non-null if and only if $p\Phi_p(y)=0$, while $y \in \mu_p(R)$ is non-null  
if and only if $\Phi_p(y)=0$. These are equivalent conditions when $p$ is invertible in $R$, but not in 
general. 
\end{remark}

\section{Primitivity of points in truncated $p$-divisible groups}

In this section we fix a prime number $p$. 

\subsection{Definition of primitivity} \label{sub.Precise} 

Now we consider a $p$-divisible group $\mathcal{G}$ of height $h$ over any base scheme $S$. For any positive 
integer $i$ we are interested in the $p^i$-torsion $\mathcal{G}[p^i]$ in $\mathcal{G}$, but henceforth we  abbreviate $\mathcal{G}[p^i]$ to $\mathcal{G}_i$. 
 For any pair $i,j$ of positive integers 
there is then a short exact (in the fppf sense) sequence 
\begin{equation}
0 \to \mathcal{G}_i \hookrightarrow \mathcal{G}_{i+j} \twoheadrightarrow \mathcal{G}_j \to 0. 
\end{equation}
The arrow $\mathcal{G}_{i+j}\twoheadrightarrow \mathcal{G}_j$ (strictly speaking, its composition with $\mathcal{G}_j \hookrightarrow \mathcal{G}_{i+j}$) 
is given by raising to the power $p^i$, and it is finite locally free of rank $p^{hi}$.

Let $R$ be a $k$-algebra, 
let $x$ be an $R$-valued point of $\mathcal{G}_i$, and write $\bar x$ for the image of $x$ under the 
canonical homomorphism $\mathcal{G}_i \twoheadrightarrow \mathcal{G}_1$ (raising to the power $p^{i-1}$). 
We say that $x$ is \emph{primitive} if $\bar x$ is non-null in $\mathcal{G}_1(R)$.  

In other words, if we define $\mathcal{G}_i^{\prim}$ as the fiber product making 
\begin{equation}
\begin{CD}
\mathcal{G}_i^{\prim} @>>> \mathcal{G}_i \\
@VVV @VVV \\
\mathcal{G}_1^{\times} @>>> \mathcal{G}_1
\end{CD}
\end{equation}
cartesian, then $x \in \mathcal{G}_i(R)$ is primitive if and only if it lies in the image of the $R$-points 
of  $\mathcal{G}_i^{\prim}$. Because the square is cartesian, we see that 
\begin{itemize}
\item
$\mathcal{G}_i^{\prim} \to \mathcal{G}_i$ is a closed immersion, and 
\item 
$\mathcal{G}_i^{\prim} \to \mathcal{G}_1^{\times}$ is finite, locally free of rank $p^{h(i-1)}$. 
\end{itemize}
Now $\mathcal{G}_1^{\times}$ is finite, locally free of rank $p^h-1$ over $S$, so we conclude that
$\mathcal{G}_i^{\prim}$ is finite, locally free of rank $(p^h-1)p^{h(i-1)}$ over $S$.

\subsection{Comparison with points of exact order $N$ on elliptic curves} 

In this subsection we fix $i$ and put  $N = p^i$. 
Consider an elliptic curve $E$ over $S$.
 Let $E_N$ denote its $N$-torsion points. 
 Then consider the following  two closed subschemes of $E_N$, namely 
\begin{itemize}
\item
the closed subscheme $E_N^{\prim}$ defined above, and 
\item  the closed subscheme, call it 
$E_N^\sharp$, of points of 
exact order $N$ in the sense of Drinfeld and Katz-Mazur (see \cite{D,KM}). 
\end{itemize}
We claim that 
\begin{itemize}
\item[($\star$)]
$E_N^{\prim}$ coincides with $E_N^\sharp$.
\end{itemize}

We need to prove that ($\star$) holds for every elliptic curve $E/S$. We cannot see a priori a natural morphism between these two objects, so we proceed in the same way that similar 
problems are treated in \cite{KM}. \Preston{Added a remark that we don't know of a natural morphism between the two objects}

{\bf Step 1}  It is evident that ($\star$) holds when $p$ is invertible on $S$, because    
 $E_N/S$ is then \'etale.

{\bf Step 2} Next we check that  ($\star$) holds for $E/S$ whenever $S$ is flat over $\mathbb Z$. 
In this situation $E_N$, $E_N^{\prim}$, and $E_N^\sharp$ are flat over $\mathbb Z$. 
Here we used that $E_N/S$ is flat (standard), that $E_N^{\prim}/S$ is flat (see subsection \ref{sub.Precise}), and that 
$E_N^\sharp/S$ is flat (see \cite[Thm.~5.1.1]{KM}).
By Step 1 the closed subschemes $E_N^{\prim}$ and $E_N^\sharp$ of $E_N$ coincide 
over the locus in $S$ where $p$ is invertible, so the flatness of  $E_N$, $E_N^{\prim}$ and $E_N^\sharp$ 
 over $\mathbb Z$ forces 
 $E_N^{\prim}$ to coincide with $E_N^{\sharp}$ over all of $S$. 

{\bf Step 3} Let $\mathcal E$ denote the moduli stack (over $\mathbb Z$) of elliptic curves, and choose a presentation 
(see \cite{LM}) 
$f : \mathcal M \twoheadrightarrow \mathcal E$ for it. Here $f$ is \'etale and surjective, and $\mathcal M$ 
is a smooth scheme of finite type over $\mathbb Z$. 
Pulling back the universal elliptic curve on $\mathcal E$, we 
obtain an elliptic curve $\mathbf E$ on the scheme $\mathcal M$. 
In the the terminology of \cite{KM}, $\mathbf E/\mathcal M$ is a ``modular family.'' 

Now consider an elliptic curve $E$ over an arbitrary base scheme $S$. We consider the product $\mathcal M \times S$ and 
write $p_1$, $p_2$ for the two projections. We then have two elliptic curves over $\mathcal M \times S$, 
namely $p_1^*\mathbf E$ and $p_2^*E$, and we form the $\mathcal M \times S$-scheme $T$ of isomorphisms 
between $p_1^*\mathbf E$ and $p_2^*E$. Over $T$ the elliptic curves $\mathbf E$ and $E$ become tautologically isomorphic; the resulting elliptic curve on $T$ will be denoted $\tilde E$. 

At this point we have a commutative diagram 
\begin{equation}
\begin{CD}
\mathbf E @<<< \tilde{E} @>>> E \\
@VVV @VVV @VVV \\
\mathcal M @<<< T @>>> S
\end{CD}
\end{equation}
in which both squares are cartesian. 
The two arrows in the bottom row exhibit $T$ as the fiber product of $\mathcal M$ and $S$ over $\mathcal E$, 
so $T \to S$ is \'etale and surjective. 

Now $\mathcal M$ is flat over $\mathbb Z$, so ($\star$) holds for $\mathbf E/\mathcal M$. It follows that
($\star$) holds for $\tilde{E}/T$. Here we used that the operations of forming $E_N^{\prim}$ and $E_N^\sharp$ both 
commute with base change (use subsection \ref{sub.BaseChange} and the first chapter of \cite{KM}, especially their 
Corollary 1.3.7). So $E_N^{\prim}$ and  $E_N^\sharp$ become equal after the \'etale surjective base change
$T \to S$. By descent theory $E_N^{\prim}$ and $E_N^\sharp$ are themselves equal.

\bibliographystyle{amsalpha}

\bibliography{newbib}

\end{document}